\documentclass[a4paper,12pt,english]{amsart}

\usepackage{amsmath,amsfonts,amssymb,%
	mathtools,esint,mathrsfs}
\usepackage[T1]{fontenc}
\usepackage[utf8]{inputenc}
\usepackage{babel,amsthm}
\usepackage[unicode=true,pdfusetitle,%
	bookmarks=true,bookmarksnumbered=false,%
	bookmarksopen=false,breaklinks=false,%
	pdfborder={0 0 0},backref=false,colorlinks=false]{hyperref}

\def\Q{\mathbb{Q}}
\def\R{\mathbb{R}}
\def\N{\mathbb{N}}
\def\L{\mathscr{L}}
\def\varL{\mathcal{L}}
\def\varalpha{\mathfrak{a}}
\def\varbeta{\mathfrak{b}}

\newtheorem{Theorem}{Theorem}

\newtheorem{Lemma}[Theorem]{Lemma}

\title[The Least Prime Number in a Beatty Sequence]{The Least Prime Number in a Beatty Sequence}


\author{Jörn Steuding}
\address{Jörn Steuding\\%
	Department of Mathematics\\%
	University of Würzburg\\%
	Campus Hubland Nord\\%
	Emil-Fischer-Straße 40\\%
	D-97074 Würzburg\\%
	Germany}
\email{steuding@mathematik.uni-wuerzburg.de}

\author{Marc Technau}
\address{Marc Technau\\%
	Department of Mathematics\\%
	University of Würzburg\\%
	Campus Hubland Nord\\%
	Emil-Fischer-Straße 40\\%
	D-97074 Würzburg\\%
	Germany}
\email{marc.technau@mathematik.uni-wuerzburg.de}

\date{November 2015, revised May 2016}

\begin{document}

\begin{abstract}
We prove an upper bound for the least prime in an irrational Beatty sequence. This result may be compared with Linnik's theorem on the least prime in an arithmetic progression.
\end{abstract}

\maketitle

\section{Introduction and Statement of the Main Results}

In 1944, Yuri~Linnik \cite{linnik1,linnik2} showed that for every sufficiently large $q$ and coprime $a$, there exists a constant $\ell$ such that the least prime $p$ in the prime residue class $a\bmod\,q$ satisfies $p\leq cq^\ell$, where $c$ is another absolute constant; the best bound so far for the appearing constant in the exponent is $\ell\leq 5$ due to Triantafyllos~Xylouris \cite{xylouris}. This question dates back to  Savardaman Chowla \cite{chowla} who conjectured in 1934 that one may even take $\ell=1+\epsilon$. Two years later P\'al Tur\'an \cite{turan} proved that this is true for almost all $q$ and that it holds in general under the assumption of the generalized Riemann hypothesis. In this note we shall investigate the corresponding question about the least prime in a Beatty sequence.
\par

Given a positive real number $\alpha$ and a non-negative real $\beta$, the associated (generalized) Beatty sequence is defined by
$$
{\mathcal B}(\alpha,\beta)=\{\lfloor n\alpha+\beta\rfloor\,:\,n\in\N\},
$$
where $\lfloor x\rfloor$ denotes the largest integer less than or equal to $x$. If $\alpha$ is rational, then ${\mathcal B}(\alpha,\beta)$ is a union of residue classes and, if at least one of them is a prime residue class, we may apply Linnik's theorem to bound the least prime (see the concluding remarks, \S~\ref{concluding}). Otherwise, if $\alpha$ is irrational, ${\mathcal B}(\alpha,\beta)$ does not contain an entire residue class. It follows from a classical exponential sum estimate due to Ivan~M.~Vinogradov \cite{vinogradov1} that there exist infinitely many prime numbers in such a Beatty sequence (details in \S~\ref{irrational}), hence, in particular there exists a least prime number.
However, the problem of estimating the size of the least prime is clearly different as compared to the rational case. In fact, for an integer $m\geq 2$ consider $\alpha_m=4+\sqrt{2}/m$. Since $\lfloor n\alpha_m\rfloor=4n$ for $n=1,2,\ldots,\lfloor m/\sqrt{2}\rfloor=:M$, there is no prime amongst the first $M$ elements of ${\mathcal B}(\alpha_m,0)$. It turns out that the diophantine character of $\alpha$ has an impact on a non-trivial bound for the least prime in the associated Beatty sequence:

\begin{Theorem}\label{uno}
For every positive $\epsilon$ there exists a computable positive integer $\ell$ such that for every irrational $\alpha>1$ the least prime $p$ in the Beatty sequence ${\mathcal B}(\alpha,\beta)$ satisfies the inequality
\begin{equation}\label{estimate1}
p \leq \varL^{35-16\epsilon} \alpha^{2(1-\epsilon)} B p_{m+\ell}^{1+\epsilon},
\end{equation}
where
$B=\max\{1,\beta\}$, $\varL=\log(2\alpha B)$,
$p_n$ denotes the numerator of the $n$-th convergent to the regular continued fraction expansion of $\alpha=[a_0,a_1,\ldots]$ and $m$ is the unique integer such that
\begin{equation}\label{m}
p_m \leq \varL^{16} \alpha^2 < p_{m+1}.
\end{equation}
\end{Theorem}

\noindent Since the sequence of numerators $p_n$ is recursively given by 
\begin{equation}\label{recur}
p_{-1}=1,\ p_0=a_0=\lfloor \alpha\rfloor,\quad \mbox{and}\quad p_{n+1}=a_{n+1}p_n+p_{n-1}
\end{equation}
with partial quotients $a_n\in\N$ for $n\in\N$, the $p_n$'s are strictly increasing integers and the integer $m$ satisfying~\eqref{m} is uniquely determined. In view of the explicit exponents and absolute constants $m$ and $\ell$, Theorem~\ref{uno} may be regarded as the analogue of Linnik's theorem on the least prime in an arithmetic progression. A careful analysis of the reasoning allows to assign the quantity $k$ from Theorem~\ref{uno} an explicit (possibly not optimal) value:

\begin{Theorem}\label{two}
On the hypothesis of Theorem~\ref{uno}, and assuming $\epsilon<\frac{44}{2025}$, the least prime $p$ in the Beatty sequence ${\mathcal B}(\alpha,\beta)$ satisfies the inequality~\eqref{estimate1}, where for $\ell$ one may take the least positive integer satisfying
\begin{equation}\label{ell:ineq}
	\ell \geq 3 + 9\epsilon^{-1} \left(
		41 + \log\left(1+\epsilon^{-1}\right)
		+ \log\left(3711+2\cdot17^{-3}M_{\frac{5}{4}\epsilon}\right) 
	\right),
\end{equation}
where $M_\epsilon$ is a constant such that~\eqref{eq:DivisorBound:General} holds.
\end{Theorem}

\section{Proof of Theorem~\ref{uno}}\label{irrational}

In the process of proving the ternary Goldbach conjecture for sufficiently large odd integers Vinogradov \cite{vinogradov1} obtained the estimate
$$
\sum_{p\leq N}\exp(2\pi im{\varalpha}p)\ll_\epsilon N^{1+\epsilon}\left({\frac{1}{N^{1/2}}}+{\frac{q}{N}}+{\frac{m}{q}}+{\frac{m^4}{q^2}}\right)^{1/2},
$$
where $m,q,N$ are positive integers, the exponential sum is taken over all prime numbers $p$ less than or equal to $N$, and $q$ is related to ${\varalpha}$ by the existence of an integer $a$ such that 
\begin{equation}\label{dirichlet}
\left\vert{\varalpha}-{a\over q}\right\vert<{1\over q^2}.
\end{equation}
For irrational ${\varalpha}$ Vinogradov's bound is $o(\pi(N))$, where $\pi(N)$ denotes the number of primes $p\leq N$. Letting ${\varalpha}={1\over \alpha}$, this implies that the sequence of numbers $\alpha p$, where $p$ runs through the prime numbers in ascending order, is uniformly distributed modulo one, answering a question of the young Paul Erd\H{o}s (cf. Tur\'an \cite{turan}, p. 227). This means that the proportion of fractional parts $\{\alpha p\}:=\alpha p-\lfloor \alpha p\rfloor$ which fall in an interval $[a,b)\subset [0,1)$ is equal to $b-a$ (the length of the interval) as follows from a well-known criterion due to Hermann Weyl \cite{weyl} stating that a sequence of real numbers $x_n$ is uniformly distributed modulo one if and only if, for every integer $h\neq 0$,
$$
\lim_{N\to\infty}{1\over N}\sum_{n\leq N}\exp(2\pi i hx_n)=0.
$$
Notice that an integer $m$ lies in ${\mathcal B}(\alpha,\beta)$ if and only if $m=\lfloor n\alpha+\beta\rfloor$ for some $n\in\N$. Equivalently, the inequalities 
\begin{equation}\label{new}
n\alpha+\beta-1<m\leq n\alpha+\beta
\end{equation}
hold. In order to find a prime number $p$ in ${\mathcal B}(\alpha,\beta)$ we thus need
\begin{equation}\label{palpha}
{p\over \alpha}\in \left({\beta-1\over \alpha},{\beta\over \alpha}\right]\ \bmod\,1,
\quad \mbox{and} \quad p > \alpha+\beta-1.
\end{equation}
Here the right hand-side is to be interpreted modulo one, and if there lies an integer in $({\beta-1\over \alpha},{\beta\over \alpha})$ the set consists of two disjoint intervals; in any case the Lebesgue measure of the set on the right equals ${1\over \alpha}$. In view of Vinogradov's aforementioned uniform distribution result the number $\pi_{{\mathcal B}(\alpha,\beta)}(N)$ of primes $p\in{\mathcal B}(\alpha,\beta)$ with $p\leq N$ satisfies  
$$
\pi_{{\mathcal B}(\alpha,\beta)}(x)\sim {1\over \alpha}\pi(x).
$$
Already Vinogradov provided an error term estimate here. However, for our purpose we shall use the following theorem of Robert C. Vaughan \cite{vaughan}: {\it Let ${\varalpha}\in\R$ and suppose that $a$ and $q$ are coprime integers satisfying~\eqref{dirichlet}. Moreover, for $0<\delta<{1\over 2}$, define 
\begin{equation}\label{eq:chi:def}
\chi_\delta(\theta)=\left\{\begin{array}{c@{\quad}l} 1 & \mbox{if}\ -\delta < \theta \leq \delta,\\ 
0 & \mbox{if either}\ -{1\over 2}\leq \theta \leq -\delta\ \mbox{or}\ \delta < \theta \leq {1\over 2},\end{array}\right.
\end{equation}
and to be periodic with period $1$. Then, for arbitrary real ${\varbeta}$, every positive integer $N$, and any real $\epsilon>0$,}
$$
\sum_{n\leq N}\Lambda(n)\left(\chi_\delta(n{\varalpha}+{\varbeta})-2\delta\right)\ll_\epsilon {\L}^8\left({N\over q^{1\over 2}}+N^{3\over 4}+(N q\delta)^{1\over 2}+\delta^{2\over 5}N^{4\over 5}\left({Nq\over \delta}\right)^\epsilon\right)
$$
{\it with} ${\L}:=\log({Nq\over \delta})$. Here $\Lambda(n)$ is the von Mangoldt-function counting prime powers $p^\nu$ with weight $\log p$. In view of~\eqref{palpha} we shall use this with ${\varalpha}={1\over \alpha}$, ${\varbeta}={2\beta-1\over 2\alpha}$ and $\delta={1\over 2\alpha}$. This leads to
\begin{equation}\label{eq}
\sum_{\substack{p^\nu\leq N \\ p^\nu\in{\mathcal B}(\alpha,\beta)}}\log p
+\sum_{\substack{p^\nu<\alpha+\beta-1 \\ p^\nu\in{\mathcal B}(\alpha,\beta-\lfloor\alpha+\beta\rfloor)}}\log p
={1\over \alpha} \sum_{p^\nu\leq N}\log p+{\mathcal E}_\alpha(N,q),
\end{equation}
where 
\begin{equation}\label{E}
\vert {\mathcal E}_\alpha(N,q)\vert \leq c{\L}^8\left({N\over q^{1\over 2}}+N^{3\over 4}+\left({N q\over 2\alpha}\right)^{1\over 2}+{N^{4\over 5}\over (2\alpha)^{2\over 5}}(2Nq\alpha)^\epsilon \right)
\end{equation}
with ${\L}=\log(2Nq\alpha)$ and appropriate absolute constant $c$ depending only on $\epsilon$, but not on $\alpha$. The second sum on the left hand side of~\eqref{E} may be estimated using a classical inequality due to John B. Rosser \& Lowell Schoenfeld \cite{rs} (see Lemma~\ref{lem:RosserSchoenfeld:ChebyshevFnc} below). The number of prime powers $p^\nu\leq N$ with $\nu\geq 2$ is less than or equal to $\pi(N^{1\over 2})$, hence we may replace~\eqref{eq} by 
$$
\sum_{\substack{p\leq N \\ p\in{\mathcal B}(\alpha,\beta)}}\log p \geq {1\over \alpha} \sum_{p\leq N}\log p+{\mathcal E}_\alpha(N,q) - 1.04(\alpha+\beta-1)
+\left({1\over \alpha}-1\right)\sum_{\substack{p^\nu\leq N \\ \nu\geq 2}}\log p.
$$
Notice that the last term is negative; it is obviously bounded by 
$$
\left(1-{1\over \alpha}\right)\sum_{\substack{p^\nu\leq N\\ \nu\geq 2}}\log p< \left(1-{1\over \alpha}\right)\pi(N^{1\over 2})\log N< \left(1+{3\over \log N}\right)N^{1\over 2},
$$
where we have used a classical inequality for the prime counting function $\pi(x)$ valid for all $x$ also due to Rosser \& Schoenfeld \cite{rs}. Using another one of their explicit inequalities, namely
$$
\sum_{p\leq N}\log p>N-{N\over \log N}\qquad \mbox{for}\quad N\geq 41, 
$$
we thus find a prime $p\leq N$ in ${\mathcal B}(\alpha,\beta)$ if we can show that 
$$
{N\over \alpha}\left(1-{1\over \log N}\right)> \vert {\mathcal E}_\alpha(N,q)\vert+\left(1+{3\over \log N}\right)N^{1\over 2} + 1.04(\alpha+\beta-1),
$$
which we may also replace by
$$
0.73 {N\over \alpha} > \vert {\mathcal E}_\alpha(N,q)\vert + 1.81 N^{1\over 2}
+ 1.04(\alpha+\beta-1).
$$
By~\eqref{E} this inequality is satisfied if
\begin{align*}
0.73>&1.81{\alpha\over N^{1\over 2}} + 1.04{{\alpha^2+\alpha\beta-\alpha} \over N} \\
&+c{\L}^8\left({\alpha\over q^{1\over 2}}+{\alpha\over N^{1\over 4}}+\left({q\alpha\over 2N}\right)^{1\over 2}+{\alpha^{3\over 5}\over N^{1\over 5}}(2Nq\alpha)^\epsilon \right).
\end{align*}
Obviously, $N$ needs to be larger than $\max\{\alpha^4,\beta\}$ and $q$ larger than $\alpha^2$. Since $\L$ depends on $\alpha$ and $\beta$ and we would like to eliminate this dependency, we shall take both $N$ and $q$ somewhat larger. Therefore, we make the ansatz $N=\varL^{35}\alpha^2 B q \eta^\epsilon$ and $q=\varL^{16} \alpha^2 \eta$ with some large parameter $\eta$, to be specified later, and $B=\max\{1,\beta\}$. (The exponent $35$ was chosen with hindsight from both this proof and the proof of Theorem~\ref{two}; from the above inequality we see that it should be at least $16$ and the further increase is due to the positive contribution from the $\epsilon$-term.) Then, the latter inequality can be rewritten as
\begin{align}\label{E:specialChoice}
\begin{split}
0.73>& 1.81 \varL^{-\frac{51}{2}} \alpha^{-1} B^{-\frac{1}{2}} \eta^{-\frac{1+\epsilon}{2}}
	+ 1.04 \varL^{-51} (\alpha+\beta-1) \alpha^{-3} B^{-1} \eta^{-(1+\epsilon)} \\
	&+ c \L^8 \Bigl(
		\varL^{-8} \eta^{-\frac{1}{2}}
		+ \varL^{-\frac{51}{4}} B^{-\frac{1}{4}} \eta^{-\frac{1+\epsilon}{4}}
		+ 2^{-\frac{1}{2}} \varL^{-\frac{35}{2}}
			\alpha^{-\frac{1}{2}} B^{-\frac{1}{2}} \eta^{-\frac{\epsilon}{2}} \\
		&+ 2^\epsilon \varL^{-\frac{51}{5}+67\epsilon}
			\alpha^{-\frac{1}{5}+7\epsilon} B^{-\frac{1}{5}+\epsilon}
			\eta^{-\frac{1+\epsilon}{5} + (2+\epsilon)\epsilon}
	\Bigr),
\end{split}
\end{align}
where ${\L}=\log(2 \varL^{67} \alpha^7 B \eta^{2+\epsilon})$. Assuming $\epsilon < {1\over 35}$, as we may, all exponents belonging to $\alpha$, $B$ and $\eta$ are negative. Since $\L \varL^{-1} \ll_\kappa \eta^{-\kappa}$ for any $\kappa>0$, the above inequality is fulfilled for all sufficiently large $\eta$, say $\eta\geq\eta_0$.

Since $\eta$ is intertwined with $q$ a little care needs to be taken.
In order to find a suitable $\eta$ recall that $\alpha$ is irrational. Hence, by Dirichlet's approximation theorem, there are infinitely many solutions ${a\over q}$ to inequality~\eqref{dirichlet}; in view of ${\varalpha}={1\over \alpha}$ we may take for ${a\over q}$ the reciprocals of the convergents ${p_n\over q_n}$ to the continued fraction expansion of $\alpha$. (For this and further fundamental results about diophantine approximation and continued fractions we refer to Hardy \& Wright \cite{hw}.) We shall choose $\ell$ such that $\eta_0\leq {p_{m+\ell}\over p_m}$, where $m$ is defined by~\eqref{m}, for then the choice $q=p_{m+\ell}$ will yield an $\eta\geq\eta_0$. In fact, if $F_n$ denotes the $n$-th Fibonacci number (defined by the recursion $F_{n+1}=F_n+F_{n-1}$ and the initial values $F_0=0$ and $F_1=1$), we observe by~\eqref{recur} and induction that
$$
p_{m+\ell}\geq p_{m+\ell-1}+p_{m+\ell-2}\geq \ldots \geq F_{j+1}p_{m+\ell-j}+F_jp_{m+\ell-(j+1)}
$$
for $j\leq m+\ell-1$. In view of Binet's formula,
$$
F_n={\textstyle{1\over \sqrt{5}}}(G^n+(-G)^{-n}),
$$ 
where $G={1\over 2}(\sqrt{5}+1)$ is the golden ratio, we thus find 
\begin{equation}\label{eq:eta:ell:ineq}
\eta \geq {p_{m+\ell}\over p_{m+1}}\geq F_{\ell}\geq {\textstyle{1\over \sqrt{5}}}G^{\ell-1}.
\end{equation}
For sufficiently large $\ell$ the right hand side of this inequality will exceed $\eta_0$ and, hence, we obtain the desired choices to satisfy~\eqref{E:specialChoice}. This completes the proof of Theorem~\ref{uno}. 

\section{Proof of Theorem~\ref{two}}\label{explicit}

In order to make the inequalities from the previous section effective, we require an effective version of Vaughan's theorem. In order to state this we first introduce some notation: let $d_k(x)$ denote the number of representations of the positve integer $x$ as a product of exactly $k$ positive integers. When $k=2$ we also write $d(x)=d_2(x)$. It is well-known that for every $\epsilon>0$ there is a constant $M_\epsilon>0$ such that
\begin{equation}
d(x)\leq M_{\epsilon}x^{\epsilon}.\label{eq:DivisorBound:General}
\end{equation}
Indeed, one may take
\begin{equation}
M_{\epsilon}=\left(\frac{2}{e\log2}\right)^{e^{1/\epsilon}}\label{eq:DivisorBound:Vinogradov}
\end{equation}
(see, e.g., \cite[p. 38]{vinogradov1985selected}). Now we state:
\begin{Theorem}[\cite{vaughan}]
\label{thm:Vaughan:2:Expl}Let $\varalpha\in\R$ and suppose that $a$ and $q$ are coprime integers satisfying~\eqref{dirichlet}. Suppose that $0<\delta<\frac{1}{2}$,
$\mathscr{L}\coloneqq\log(Nq/\delta)$ and let $\chi_\delta$ be given by~\eqref{eq:chi:def}. Then, for any $\epsilon>0$ and $M_\epsilon$ satisfying~\eqref{eq:DivisorBound:General},
\begin{align}
\biggl|\sum_{n\leq N}\Lambda(n)(\chi_{\delta}(\varalpha n-\varbeta)-2\delta)\biggr| & \leq10^{3}\mathscr{L}^{8}\Bigl(3422Nq^{-\frac{1}{2}}+251N^{\frac{3}{4}}+38(\delta Nq)^{\frac{1}{2}}\label{eq:Vaughan:2:Estimate}\\
 & \quad+\left(11+17^{-3}M_{\frac{5}{4}\epsilon}\right)(Nq\delta^{-1})^{\frac{3}{4}\epsilon}\delta^{\frac{2}{5}}N^{\frac{4}{5}+\epsilon}\Bigr).\nonumber 
\end{align}
\end{Theorem}

We postpone the proof to the next section and proceed with the proof of Theorem~\ref{two}. In order to not be too repetitive we only sketch briefly what modifications are to be made. We continue with the notation introduced in the previous section. First observe that the exponents belonging to the last term in~\eqref{eq:Vaughan:2:Estimate} have changed slightly. After making the implied changes, assuming $\epsilon<\frac{44}{2025}$, and estimating the arising exponents,~\eqref{eq:Vaughan:2:Estimate} may be replaced by
\begin{align*}
0.73>& 1.81 \varL^{-\frac{51}{2}} \alpha^{-1} B^{-\frac{1}{2}} \eta^{-\frac{1+\epsilon}{2}}
	+ 1.04 \varL^{-51} (\alpha+\beta-1) \alpha^{-3} B^{-1} \eta^{-(1+\epsilon)} \\
	&+ 10^3 \L^8 \varL^{-8} \eta^{-\frac{\epsilon}{2}} \Bigl(
		3711 + 2\cdot 17^{-3}M_{\frac{5}{4}\epsilon}
	\Bigr),
\end{align*}
which, if satisfied, ensures the existence of a prime $p\leq N$ in $\mathcal{B}(\alpha,\beta)$.
Since
\[%
	\L^8 \varL^{-8} \eta^{-\frac{\epsilon}{4}} < 65^8 \left(1+\epsilon^{-1}\right)^8
\]
this is surely satisfied if
\[%
\eta^{\frac{\epsilon}{4}} > 2\cdot 10^3 65^8 \left(1+\epsilon^{-1}\right)^8
	\Bigl( 3711 + 2\cdot 17^{-3}M_{\frac{5}{4}\epsilon} \Bigr),
\]
Recalling~\eqref{eq:eta:ell:ineq} we thus find any choice of $\ell\in{\N}$, such that~\eqref{ell:ineq} holds, to be admissible. This concludes the proof of Theorem~\ref{two}.

\section{Proof of Theorem~\ref{thm:Vaughan:2:Expl}}
\label{sec:ProofOfVaughanThm}

The present section is devoted to a proof of Theorem~\ref{thm:Vaughan:2:Expl}. To this end,
we follow Vaughan \cite{vaughan} and replace his estimates with effective
inequalities.

\subsection{Outline of the argument}

The basic idea is to use a finite Fourier expansion of the function $\chi_{\delta}$
given by~\eqref{eq:chi:def} (see Lemma~\ref{lem:Harman:2.1:Fourier}).
The main term $2\delta\sum_{n\leq N}\Lambda(n)$ arises from the zeroth coefficient
in this expansion, whereas all other terms are handled as error, which reduces the
problem to estimating sums of the type
\[
\sum_{h}\biggl|\sum_{n}\Lambda(n)e(\alpha hn)\biggr|,
\]
with variables varying in between suitable ranges.
This is achieved by applying Vaughan's identity to $\Lambda(n)$,
yielding sums of the type
\[
\sum_{h}\biggl|\sum_{x}\sum_{y}\sum_{z}a_{x}b_{y}e(\alpha hxyz)\biggr|.
\]
These are handled by combining variables in a suitable way and applying
some estimates on exponential sums.

\subsection{The Piltz divisor function}
Since combining variables naturally introduces divisor functions to the coefficients in the
new sums we shall need explicit bounds for these. Concerning bounds of the type
\eqref{eq:DivisorBound:General}, we have already noted that the choice
\eqref{eq:DivisorBound:Vinogradov} is admissible for every $\epsilon>0$. However,
for most arguments in the present section we need three special instances, namely
$\epsilon\in\{\frac{1}{2},\frac{1}{4},\frac{1}{6}\}$. Therefore we would like to
do somewhat better in these cases.

Indeed, following a well-known argument, if $x=\prod_{j}p_{j}^{\nu_{j}}$
with distinct primes $p_{j}$ then it holds that
\[
\frac{d(x)}{x^{\epsilon}}=\biggl\{\prod_{p_{j}\geq e^{1/\epsilon}}\cdot\prod_{p_{j}<e^{1/\epsilon}}\biggr\}\frac{\nu_{j}+1}{p_{j}^{\epsilon\nu_{j}}}\eqqcolon\Pi_{1}\cdot\Pi_{2},\quad\text{say}.
\]
If $p_{j}\geq e^{1/\epsilon}$ then $p_{j}^{\epsilon\nu_{j}}\geq1+\nu_{j}$
and therefore $\Pi_{1}\leq1$. As for $\Pi_{2}$ just note that
\[
\Pi_{2}\leq\prod_{p\leq e^{1/\epsilon}}\max_{\nu\in{\N}_{0}}\frac{\nu+1}{p_{j}^{\epsilon\nu}},
\]
where the last product is obviously finite and, consequently, bounded.
Calculating the above product for our three choices of $\epsilon$
yields
\begin{equation}
d(x)\leq\min\left\{ 139x^{\frac{1}{6}},\;9x^{\frac{1}{4}},\;2x^{\frac{1}{2}}\right\} .\label{eq:DivisorBound:ParticularCases}
\end{equation}

We also need the following
\begin{Lemma}
\label{lem:DivisorBounds}Suppose that $X\geq2$. Then
\begin{align*}
\sum_{x\leq X}d_{3}(x)^{2} & \leq3000X(\log X)^{8},\\
\sum_{x\leq X}d(x)^{2} & \leq7X(\log X)^{3}.
\end{align*}
\end{Lemma}
\begin{proof}
We follow the approach presented in \cite[sec. 1.6]{iwaniec2004analytic}.
We start with the inequality
\[
\sum_{x\leq X}d_{3}(x)^{2}\leq X\prod_{p\leq X}\left(1-\frac{1}{p}\right)\sum_{\nu\geq0}\frac{d_{3}(p^{\nu})^{2}}{p^{\nu}}.
\]
The sum on the right hand side can be calculated explicitly and, together
with $p\geq2$, we obtain
\[
\sum_{\nu\geq0}\frac{d_{3}(p^{\nu})^{2}}{p^{\nu}}=\left(1-\frac{1}{p}\right)^{-5}\frac{p^{2}+4p+1}{p^{2}}\leq\left(1+\frac{1}{p}\right)^{9}\left(1+\frac{1}{p^{2}}\right)^{5}.
\]
By \cite[inequality (3.20)]{rs},
\[
\prod_{p\leq X}\left(1+\frac{1}{p}\right)<\exp\left(\frac{1}{(\log X)^{2}}+\frac{1}{2}+0.26149\ldots\right)\log X,
\]
so that 
\[
\sum_{x\leq X}d_{3}(x)^{2}\leq X\prod_{p\leq X}\left(1+\frac{1}{p}\right)^{8}\exp\biggl(\sum_{p}\frac{-1}{p^{2}}\biggr)\exp\biggl(\sum_{p}\frac{5}{p^{2}}\biggr),
\]
which is $\leq3000X(\log X)^{8}$ for $X\geq6100$. For smaller values
the bound can be easily verified by a computer.

Similarly,
\[
\sum_{x\leq X}d(x)^{2}\leq X\prod_{p\leq X}\left(1+\frac{1}{p}\right)\left(1-\frac{1}{p}\right)^{4}\leq X\exp\left(\sum_{p}\frac{-1}{p^{2}}\right)\prod_{p\leq X}\left(1-\frac{1}{p}\right)^{3},
\]
which is $\leq7X(\log X)^{3}$ for $X\geq171$. Again, for smaller
$X$ the bound is verified by a computer.
\end{proof}

\subsection{Miscellany}
\label{subsec:Miscellany}
\begin{Lemma}[\cite{rs}]
\label{lem:RosserSchoenfeld:ChebyshevFnc}It holds that
\[
\sum_{n\leq N}\Lambda(n)\leq c_{0}N,
\]
for some constant $c_{0}$, where one may take $c_{0}=1.03883$.
\end{Lemma}

\begin{Lemma}
\label{lem:Harman:2.1:Fourier}Let $0<\delta<\frac{1}{2}$ and
let $L\in{\N}$. Suppose that $\chi_\delta$ is given by~\eqref{eq:chi:def}.
Then there are coefficients $c_{\ell}^{\pm}$ such that
\[
\quad|c_{\ell}^{\pm}|\leq\min\left\{ 2\delta+\frac{1}{L+1},\frac{3}{2\ell}\right\}
\]
and
\[
\chi_{\delta}(y)\lesseqgtr2\delta\pm\frac{1}{L+1}+\sum_{0<|\ell|\leq L}c_{\ell}^{\pm}e(\ell y).
\]
\end{Lemma}
\begin{proof}
See \cite[Lemma 2.1]{harman2007prime}, resp. \cite[pp. 18--21]{baker1986diophantine}.
\end{proof}

\subsection{Bounds on some exponential sums}

\label{subsec:ExpSumBounds}
\begin{Lemma}[\cite{vaughan}]
\label{lem:Vaughan:1:Expl}Suppose that $X\geq1$, $Y\geq1$ and
let $\beta\in{\R}$. Then
\begin{align}
\sum_{x\leq X}\min\left\{ Y,\frac{1}{2\left\Vert \alpha x-\beta\right\Vert }\right\}  & <4XYq^{-1}+4Y+(X+q)\log q,\label{eq:Vaughan:(7):Expl}\\
\sum_{x\leq X}\min\left\{ \frac{XY}{x},\frac{1}{2\left\Vert \alpha x\right\Vert }\right\}  & <\left(10XYq^{-1}+X+\frac{7}{2}q\right)\log(2XYq),\label{eq:Vaughan:(8):Expl}
\end{align}
\end{Lemma}
\begin{proof}
The case when $X\leq q$ is treated in \cite[Lemma 8a and 8b]{vinogradov2004method}.
Splitting the sum in~\eqref{eq:Vaughan:(7):Expl} into $\lfloor X/q\rfloor+1$
sums of length at most $q$ and applying \cite[Lemma 8a]{vinogradov2004method}
yields the first inequality. The second inequality can be obtained
by splitting the sum in~\eqref{eq:Vaughan:(8):Expl} in essentially
the same way as in the proof of \cite[Lemma 8b]{vinogradov2004method}.
\end{proof}
\begin{Lemma}[\cite{vaughan}]
\label{lem:Vaughan:2:Expl}Let $a_{1},\ldots,a_{X},b_{1},\ldots,b_{Y}$
be complex numbers and write $l=\log(2XYq)$. Then
\begin{align}
\sum_{x\leq X}\max_{Z\leq XY/x}\biggl|\sum_{y\leq Z}a_{x}e(\alpha xy)\biggr| & \leq l\left(10XYq^{-1}+X+\frac{7}{2}q\right)\max_{x\leq X}|a_{x}|.\label{eq:lem:Vaughan:3:Expl}\\
\sum_{x\leq X}\max_{Z\leq Y}\biggl|\sum_{y\leq Z}a_{x}b_{y}e(\alpha xy)\biggr| & \leq l^{\frac{3}{2}}\biggl(\sum_{x\leq X}|a_{x}|^{2}\sum_{y\leq Y}|b_{y}|^{2}\biggr)^{\frac{1}{2}}\label{eq:Vaughan:S:Estim}\\
 & \quad\cdot(167XYq^{-1}+70X+6Y+10q)^{\frac{1}{2}},\nonumber 
\end{align}
\end{Lemma}
\begin{proof}
The inequality~\eqref{eq:lem:Vaughan:3:Expl} follows directly from
\eqref{eq:Vaughan:(8):Expl}.

For the proof of~\eqref{eq:Vaughan:S:Estim} we refer the reader to
\cite{vaughan}. There is only one Vinogradov symbol
in the proof and it is hiding only a factor $2$ and a factor $\frac{1}{2}$
in the second argument of $\min(X,\left\Vert \alpha h\right\Vert ^{-1})$.
The application of (6) is to be replaced by an application of~\eqref{eq:Vaughan:(7):Expl}.
\end{proof}
\begin{Lemma}
\label{lem:Vaughan:DyadicSplitting}Suppose that $0\leq\delta<\frac{1}{2}$,
$\mathscr{L}\coloneqq\log(Nq/\delta)$ and
\begin{equation}
J\leq J'\leq H\leq q\leq N,\qquad J'<2J.\label{eq:Vaughan:1:Assumptions}
\end{equation}
Then
\begin{align*}
\sum_{J\leq j<J'}\Biggl|\sum_{n\leq N}\Lambda(n)e(\alpha jn)\Biggr| & \leq10^{3}\mathscr{L}^{7}\biggl(560JNq^{-\frac{1}{2}}+41JN^{\frac{3}{4}}+86(JNq)^{\frac{1}{2}}\\
 & \quad+\left(21+10^{-7}M_{\frac{5}{4}\epsilon}\right)J^{\frac{3}{5}+\frac{3}{4}\epsilon}N^{\frac{4}{5}+\epsilon}\biggr).
\end{align*}
\end{Lemma}
\begin{proof}
We let 
\begin{equation}
u\coloneqq\min\left\{ N^{\frac{2}{5}}J^{-\frac{1}{5}},q,Nq^{-1}\right\} \label{eq:Vaughan:(13):uChoice}
\end{equation}
and use Vaughan's identity,
\[
\sum_{n\leq N}\Lambda(n)e(\alpha jn)=S_{1,j}-S_{2,j}+S_{3,j}-S_{4,j},
\]
where
\begin{align*}
S_{1,j} & =\sum_{n\leq u}\Lambda(n)e(\alpha jn),\\
S_{2,j} & =\underset{d,n\leq u}{\sum\sum}\sum_{r\leq N/dn}\mu(d)\Lambda(n)e(\alpha jdrn),\\
S_{3,j} & =\sum_{d\leq u}\sum_{n\leq N/d}\sum_{r\leq N/dn}\mu(d)\Lambda(n)e(\alpha jdrn),\\
S_{4,j} & =\sum_{u<m<N/u}\sum_{u<n\leq N/m}\sum_{\substack{d\mid m\\
d\leq u
}
}\mu(d)\Lambda(n)e(\alpha jmn),
\end{align*}
so that it remains to bound the contribution from $\sum_{J\leq j<J'}|S_{k,j}|$
for $k=1,2,3,4$. For $k=1$ this contribution is $\leq c_{0}N^{\frac{1}{2}}$.

To treat the case $k=2$ we combine the variables $jdn=x$ and employ
\eqref{eq:lem:Vaughan:3:Expl}, getting
\[
\sum_{J\leq j<J'}|S_{2,j}|\leq\mathscr{L}^{2}\left(8J'Nq^{-1}+\frac{4}{5}J'u^{2}+\frac{14}{5}q\right)\max_{x\leq J'u^{2}}\sum_{\substack{j\mid x\\
j\leq J'
}
}1.
\]
In view of~\eqref{eq:Vaughan:1:Assumptions},~\eqref{eq:Vaughan:(13):uChoice}
and~\eqref{eq:DivisorBound:ParticularCases},
\[
\max_{x\leq J'u^{2}}\sum_{\substack{j\mid x\\
j\leq J'
}
}1\leq\min\left\{ 139q^{\frac{1}{2}},\:3M_{\frac{5}{4}\epsilon}(J^{\frac{3}{4}}N)^{\epsilon},\:11(JNq^{-1})^{\frac{1}{2}}\right\} ,
\]
so that
\[
\sum_{J\leq j<J'}|S_{2,j}|\leq\mathscr{L}^{2}\left(2224JNq^{-\frac{1}{2}}+3M_{\frac{5}{4}\epsilon}J^{\frac{3}{5}+\frac{3}{4}\epsilon}N^{\frac{4}{5}+\epsilon}+31(JNq)^{\frac{1}{2}}\right).
\]

The case $k=3$ is handled by combining $jd=x$ and $rn=y$ to obtain
\[
\sum_{J\leq j<J'}|S_{3,j}|\leq\sum_{x<J'u}\biggl(\sum_{\substack{j\mid x\\
j\leq J'
}
}1\biggr)\max_{Z\leq J'N/x}\biggl|\sum_{y\leq Z}(\log y)e(\alpha xrn)\biggr|
\]
and partial summation to remove the $\log$-factor together with~\eqref{eq:Vaughan:(8):Expl}
and~\eqref{eq:Vaughan:1:Assumptions} yield
\begin{align*}
\sum_{J\leq j<J'}|S_{3,j}| & \leq\mathscr{L}^{2}\left(20J'Nq^{-1}+2J'u+7q\right)\max_{x\leq J'u}\sum_{\substack{j\mid x\\
j\leq J'
}
}1\\
 & \leq\mathscr{L}^{2}\left(360JNq^{-\frac{1}{2}}+44JN^{\frac{1}{2}}+21(JNq)^{\frac{1}{2}}\right),
\end{align*}
where for the last inequality we have used that
\[
\max_{x\leq J'u}\sum_{\substack{j\mid x\\
j\leq J'
}
}1\leq\min\left\{ 9q^{\frac{1}{2}},\:11(J^{2}N)^{\frac{1}{10}},\:3(JNq^{-1})^{\frac{1}{2}}\right\} .
\]

For the previous two cases we did combine $j$ with small variables
($\leq u$), which does not immediately work in case $k=4$. If $m\ll N^{\frac{1}{2}}$
we combine $j$ and $m$ for otherwise we have $n\ll N^{\frac{1}{2}}$
and combine $j$ with $n$ instead. In either case we run into the
problem that the range of $n$ is dependent on $m$. This can be resolved
by restricting the range of $m$ to short intervals.

Choose complex coefficients $\tilde{c}_{j}$ of modulus 1 such that
$\tilde{c}_{j}S_{4,j}\in{\R}_{\geq0}$. Then
\begin{align*}
\sum_{J\leq j\leq J'}|S_{4,j}| & =\sum_{J\leq j\leq J'}\sum_{u<m<N/u}\sum_{u<n\leq N/m}\tilde{c}_{j}\biggl(\sum_{\substack{d\mid m\\
d\leq u
}
}\mu(d)\biggr)\Lambda(n)e(\alpha jmn)\\
 & =\sideset{}{^{\diamondsuit}}\sum_{u\leq M<N/u}S_{M},
\end{align*}
where $\sideset{}{_{M}^{\diamondsuit}}{\textstyle \sum}$ means that
$M$ only takes values of the form $2^{k}u$ ($k\in{\N}_{0}$)
and
\[
S_{M}\coloneqq\sum_{J\leq j\leq J'}\sum_{\substack{m\sim M\\
m<N/u
}
}\sum_{u<n\leq N/m}\tilde{c}_{j}\biggl(\sum_{\substack{d\mid m\\
d\leq u
}
}\mu(d)\biggr)\Lambda(n)e(\alpha jmn),
\]
where $m\sim M$ is to be understood as $M<m\le2M$. When $u\leq M<N^{\frac{1}{2}}$,
by~\eqref{eq:Vaughan:S:Estim} and Lemma~\ref{lem:DivisorBounds}
\begin{align*}
|S_{M}| & \leq\sum_{x\sim J'M}\max_{Z\leq N/M}\biggl|\sum_{u<n\leq Z}d_{3}(x)\Lambda(n)e(\alpha xn)\biggr|\\
 & \leq128c_{0}^{\frac{1}{2}}\mathscr{L}^{6}\Bigl(2844JNq^{-\frac{1}{2}}+1841JM^{\frac{1}{2}}N^{\frac{1}{2}}\\
 & \quad+270J^{\frac{1}{2}}M^{-\frac{1}{2}}N+348(JNq)^{\frac{1}{2}}\Bigr)
\end{align*}
so that by summing over $M$ we get
\begin{align*}
\sideset{}{^{\diamondsuit}}\sum_{u\leq M<N^{1/2}}|S_{M}| & \leq\mathscr{L}^{7}\biggl(535414JNq^{-\frac{1}{2}}+\frac{480360}{\log(4N)}JN^{\frac{3}{4}}\\
 & \quad+\frac{240532}{\log(4N)}J^{\frac{1}{2}}Nu^{-\frac{1}{2}}+65753(JNq)^{\frac{1}{2}}\biggr)\\
 & \leq\mathscr{L}^{7}\biggl(\left(535414+\frac{240532}{\log(4N)}\right)JNq^{-\frac{1}{2}}+\frac{480360}{\log(4N)}JN^{\frac{3}{4}}\\
 & \quad+\frac{240532}{\log(4N)}J^{\frac{3}{5}}N^{\frac{4}{5}}+\left(65753+\frac{240532}{\log(4N)}\right)(JNq)^{\frac{1}{2}}\biggr)
\end{align*}

When $N^{\frac{1}{2}}\leq M<N/u$,
\begin{align*}
|S_{M}| & \leq\sum_{m\sim M}\max_{Z\leq J'N/M}\biggl|\sum_{Ju<y\leq Z}d(m)\biggl(\sum_{\substack{jn=y\\
J\leq j\leq J'\\
u<n\leq N/m
}
}\tilde{c}_{j}\Lambda(n)\biggr)e(\alpha my)\biggr|\\
 & \leq11\mathscr{L}^{5}\left(147JNq^{-\frac{1}{2}}+67(JMN)^{\frac{1}{2}}+20JNM^{-\frac{1}{2}}+18(JNq)^{\frac{1}{2}}\right).
\end{align*}
Hence,
\begin{align*}
\sideset{}{^{\diamondsuit}}\sum_{N^{1/2}\leq M<N/u}|S_{M}| & \leq11\mathscr{L}^{5}\biggl(107JNq^{-\frac{1}{2}}+\frac{229}{\log(4N)}J^{\frac{1}{2}}Nu^{-\frac{1}{2}}\\
 & \quad+\frac{69}{\log(4N)}JN^{\frac{3}{4}}+13(JNq)^{\frac{1}{2}}\biggr)\\
 & \leq\mathscr{L}^{5}\biggl(\left(1177+\frac{2519}{\log(4N)}\right)JNq^{-\frac{1}{2}}+\frac{2519}{\log(4N)}J^{\frac{3}{5}}N^{\frac{4}{5}}\\
 & \quad+\frac{759}{\log(4N)}JN^{\frac{3}{4}}+\left(143+\frac{2519}{\log(4N)}\right)(JNq)^{\frac{1}{2}}\biggr).
\end{align*}

The trivial bound
\[
\sum_{J\leq j<J'}\Biggl|\sum_{n\leq N}\Lambda(n)e(\alpha jn)\Biggr|\leq c_{0}N^{\frac{1}{4}}\cdot JN^{\frac{3}{4}}
\]
implies the theorem for $N\leq40000$, say. When $N>40000$ we deduce
the theorem from
\[
\sum_{J\leq j<J'}\biggl|\sum_{n\leq N}\Lambda(n)e(\alpha jn)\biggr|\leq\sum_{J\leq j<J'}\sum_{k\leq3}|S_{1,k}|+\biggl\{\:\,\sideset{}{^{\diamondsuit}}\sum_{u\leq M<N^{1/2}}+\sideset{}{^{\diamondsuit}}\sum_{N^{1/2}\leq M<N/u}\biggr\} S_{M}
\]
and all the bounds from above.
\end{proof}

\subsection{Proof of Theorem~\ref{thm:Vaughan:2:Expl}}

\label{subsec:Vaughan:2:Proof}When $N\leq28\cdot10^{22}$ we obtain
\eqref{eq:Vaughan:2:Estimate} by the trivial estimate 
\[
\biggl|\sum_{n\leq N}\Lambda(n)(\chi_{\delta}(\alpha n-\beta)-2\delta)\biggr|\leq3c_{0}N^{\frac{1}{4}}\cdot N^{\frac{3}{4}}.
\]
Subsequently we shall assume that $N>28\cdot10^{22}$. We apply Lemma
\ref{lem:Harman:2.1:Fourier} with $L=\lfloor R\delta^{-1}\rfloor$,
$R$ to be specified later, and obtain
\[
\sum_{n\leq N}\Lambda(n)(\chi_{\delta}(\alpha n-\beta)-2\delta)\lesseqgtr\pm c_{0}\frac{\delta}{R}N+\sum_{0<|\ell|\leq L}c_{\ell}^{\pm}e(-\ell\beta)\sum_{n\leq N}\Lambda(n)e(\ell\alpha n).
\]
By Abel's method of partial summation the sum on the right hand side
is found to be
\begin{equation}
\leq2\sum_{\ell\leq L}\min\left\{ 2\delta+\frac{1}{L+1},\frac{3}{2\ell}\right\} \biggl|\sum_{n\leq N}\Lambda(n)e(\ell\alpha n)\biggr|=3\frac{S(L)}{L}+3\int_{\varpi}^{L}\frac{S(u)}{u^{2}}\,\mathrm{d}u,\label{eq:PartialSummation}
\end{equation}
where
\[
S(u)\coloneqq\sum_{h\leq u}\biggl|\sum_{n\leq N}\Lambda(n)e(\alpha hn)\biggr|\quad\text{and}\quad\varpi\coloneqq\frac{3}{4\delta+\frac{2}{L+1}}\geq\frac{3}{4\delta}.
\]
Suppose for the moment that we already had a suitable bound for $S(H)$,
$\frac{3}{4\delta}\leq H\leq L$, e.g.,
\begin{align}
S(H) & \leq10^{3}\mathscr{L}^{7}\biggl(1120HNq^{-\frac{1}{2}}+82HN^{\frac{3}{4}}+294(HNq)^{\frac{1}{2}}\label{eq:S(H):inequality}\\
 & \quad+\left(62+10^{-6}M_{\frac{5}{4}\epsilon}\right)H^{\frac{3}{5}+\frac{3}{4}\epsilon}N^{\frac{4}{5}+\epsilon}\biggr).\nonumber 
\end{align}
Employing~\eqref{eq:PartialSummation}, taking $R=Nq$, and using
the estimate above gives~\eqref{eq:Vaughan:2:Estimate} as well.

In order to establish~\eqref{eq:S(H):inequality} we note that~\eqref{eq:Vaughan:S:Estim}
together with Lemma~\ref{lem:Vaughan:2:Expl} implies that
\begin{align*}
S(H) & \leq\sum_{h\leq H}\max_{Z\leq N}\Biggl|\sum_{n\leq Z}\Lambda(n)e(\alpha hn)\Biggr|\\
 & \leq(\log(2HNq))^{\frac{3}{2}}(\log N)^{\frac{1}{2}}c_{0}^{\frac{1}{2}}\\
 & \quad\cdot\left(\sqrt{167}HNq^{-\frac{1}{2}}+\sqrt{70}HN^{\frac{1}{2}}+\sqrt{6}H^{\frac{1}{2}}N+\sqrt{10}(HNq)^{\frac{1}{2}}\right).
\end{align*}
It then suffices to show~\eqref{eq:S(H):inequality} under the assumption
\eqref{eq:Vaughan:1:Assumptions} of Lemma~\ref{lem:Vaughan:DyadicSplitting}.
Applying said Lemma to the right hand side of
\[
S(H)\leq\sum_{0\leq k<\log_{2}H}\sum_{\substack{2^{k}\leq j<2^{k+1}\\
j\leq H
}
}\biggl|\sum_{n\leq N}\Lambda(n)e(\alpha jn)\biggr|
\]
yields~\eqref{eq:S(H):inequality}.

\section{Concluding Remarks}\label{concluding}

We begin with an historical remark. Beatty sequences of the form ${\mathcal B}(\alpha,{1\over 2})$ appear first in 1772 in astronomical studies of Johann Bernoulli III. \cite{bernoulli}. Elwin Bruno Christoffel \cite{christoffel} studied homogeneous Beatty seqquences ${\mathcal B}(\alpha,0)$ in 1874; those appear as well in the treatise \cite{rayleigh} by the physicist and Nobel laureate John William Strutt (Lord Rayleigh), published in 1894. The name, however, is with respect to Samuel Beatty who popularized the topic by a problem he posed in 1926 in the American Mathematical Monthly \cite{beatty1,beatty2}. The proposed problem was to show that ${\mathcal B}(\alpha,0)\cup {\mathcal B}(\alpha',0)=\N$ is a disjoint union for irrational $\alpha$ and $\alpha'$ related to one another by ${1\over \alpha}+{1\over \alpha'}=1$; this is now well-known as both, Beatty's theorem and Rayleigh's theorem. The generalization of this statement to generalized Beatty sequences has been established by Aviezri S. Fraenkel \cite{fraenkel}; already its formulation is much more complicated than in the homogeneous case $\beta=\beta'=0$.
\par

The case of Beatty sequences ${\mathcal B}(\alpha,\beta)$ with $\alpha<1$ is trivial. Since then $0\leq \lfloor (n+1)\alpha+\beta\rfloor-\lfloor n\alpha+\beta\rfloor<2$ consecutive elements differ by at most $1$, whence ${\mathcal B}(\alpha,\beta)=\N\cap [\lfloor \alpha+\beta\rfloor,+\infty)$. The case of integral $\alpha$ is trivial as well. However, we shall discuss briefly the case of rational non-integral $\alpha>1$. Given $\alpha={a\over q}$ with coprime $a$ and $q\geq 2$, it follows that ${\mathcal B}(\alpha,\beta)$ is a disjoint union of residue classes
\begin{equation}\label{decomposition}
{\mathcal B}({\textstyle{a\over q}},\beta)=\bigcup_{1\leq b\leq q} \left( \left\lfloor {ab\over q}+\beta\right\rfloor +a\N_0 \right).
\end{equation}
We observe that for some values of ${a\over q}$ and $\beta$ there is no prime residue class and even no prime number contained, e.g.,
${\mathcal B}({15\over 2},3)=\{10,18\}+15\N_0$.
It seems to be an interesting question to characterize for which parameters $\alpha\in\Q$ and $\beta$ real there appears at least one prime residue class in ${\mathcal B}(\alpha,\beta)$. In the homogeneous case, however, there is always at least one prime residue class contained in ${\mathcal B}({a\over q},0)$ which results from the choice $b=a^{-1}\bmod\,q$ in~\eqref{decomposition}.


\bigskip

\small

\begin{thebibliography}{9}

\bibitem{baker1986diophantine}{\sc R.C. {Baker}}, {\it Diophantine inequalities}, London Mathematical Society Monographs, Clarendon Press, 1986.

\bibitem{beatty1}{\sc S. Beatty}, Problem 3173, {\it Amer. Math. Monthly} {\bf 33} (1926), 159

\bibitem{beatty2}{\sc S. Beatty, A. Ostrowski, J. Hyslop, A.C. Aitken}, Solutions to Problem 3173, {\it Amer. Math. Monthly} {\bf 34} (1927), 159-160

\bibitem{bernoulli}{\sc J. Bernoulli III}, Sur une nouvelle esp\`ece de calcul, in: {\it Recueil pour les astronomes}, vols. 1, 2, Berlin 1772 

\bibitem{chowla}{\sc S. Chowla}, On the least prime in an arithmetical progression, {\it J. Indian Math. Soc.}, n. Ser., {\bf 1} (1934), 1-3

\bibitem{christoffel}{\sc E.B. Christoffel}, Observatio Arithmetica, {\it Brioschi Ann.} {\bf 6} (1873-75), 148-152

\bibitem{fraenkel}{\sc A.S. Fraenkel}, The bracket function and complementary sets of integers, {\it Canad. J. Math.} {\bf 21} (1969), 6-27

\bibitem{hw}{\sc G.H. Hardy, E.M. Wright}, {\it An introduction to the theory of numbers}, 6th ed., Oxford University Press 2007

\bibitem{harman2007prime}{\sc G.~{Harman}}, {\it Prime-detecting sieves}, Princeton University Press, 2007.

\bibitem{iwaniec2004analytic}{\sc H.~{Iwaniec} and E.~{Kowalski}}, {\it Analytic number theory}, American Mathematical Society (AMS), 2004.

\bibitem{linnik1}{\sc Yu. Linnik}, On the least prime in an arithmetic progression. I: The basic theorem, {\it Mat. Sb.}, N. Ser., {\bf 15} (1944), 139-178

\bibitem{linnik2}{\sc Yu. Linnik}, On the least prime in an arithmetic progression. II: The Deuring-Heilbronn theorem, {\it Mat. Sb.}, N. Ser., {\bf 15} (1944), 347-368

\bibitem{rs}{\sc J.B. Rosser, L. Schoenfeld}, Approximate formulas for some functions of prime numbers, {\it Ill. J. Math.} {\bf 6} (1962), 64-94

\bibitem{rayleigh}{\sc J.W. Strutt}, {\it The theory of sound. Vol. I}, 2nd ed., Macmillian, London 1894

\bibitem{turan}{\sc P. Tur\'an}, \"Uber die Primzahlen der arithmetischen Progression, {\it Acta Litt. Sci. Szeged} {\bf 8} (1937), 226-235

\bibitem{vaughan}{\sc R.C. Vaughan}, On the distribution of $\alpha p$ modulo $1$, {\it Mathematika} {\bf 24} (1977), 135-141

\bibitem{vinogradov1}{\sc I.M. Vinogradov}, A new estimate of a certain sum containing primes (Russian), {\it Rec. Math. Moscou}, n. Ser., {\bf 2}(44), No. 5, 783-792 (1937); engl. translation: New estimations of trigonometrical sums containing primes, {\it C. R. (Dokl.) Acad. Sci. URSS}, n. Ser., {\bf 17}, 165-166 (1937)

\bibitem{vinogradov1985selected}{\sc I.M. Vinogradov}, {\it Selected works}, Springer-Verlag, 1985.

\bibitem{vinogradov2004method}{\sc I.M. Vinogradov}, {\it The method of trigonometrical sums in the theory of numbers}, Dover Publications, reprint of the 1954 translation edition, 2004.

\bibitem{weyl}{\sc H. Weyl}, Sur une application de la th\'eorie des nombres \`a la m\'ecaniques statistique et la th\'eorie des pertubations, {\it L'Enseign. math} {\bf 16} (1914), 455-467

\bibitem{xylouris}{\sc T. Xylouris}, {\it \"Uber die Nullstellen der Dirichletschen L-Funktionen und die kleinste Primzahl in einer arithmetischen Progression}, Bonner Mathematische Schriften {\bf 404}, Univ. Bonn 2011, (Diss.)

\end{thebibliography}
\end{document}